\documentclass[12pt]{amsart}
\usepackage[utf8]{inputenc}
\usepackage[all]{xy}
\usepackage{amscd, amsmath, mathrsfs, amssymb, amsthm, amsxtra, bbding, epsfig, eucal, eufrak, graphicx, latexsym, url}
\usepackage{color,comment,bbm,enumitem}
\def \C {{\mathbb C}}

\def \N {{\mathbb N}}

\def \R {{\mathbb R}}

\def \Z {{\mathbb Z}}

\def \d {\,{\rm d}}
\def\re{{\Re e\,}}

\def\dm{\frac{1}{2}}

\def\sumb{\mathop{{\sum}^{\flat}}}
\def\sumstar{\mathop{{\sum}^{*}}}

\def\le{\leqslant}
\def\leq{\leqslant}
\def\ge{\geqslant}
\def\geq{\geqslant}

\def \ee{{\mathrm{e}}}

\newcommand*{\ic}{\mathrm{i}}

\newcommand*{\f}{\mathfrak{f}}
\newcommand*{\g}{\mathfrak{g}}
\newcommand*{\fh}{\mathfrak{h}}
\newcommand{\vep}{\varepsilon}

\newcommand*{\cush}{\mathfrak{S}_{\ell+1/2}}%

\newcommand*{\cC}{\mathcal{C}}

\theoremstyle{plain}
\newtheorem{theorem}{Theorem}

\newtheorem{lemma}{Lemma}[section]

\newtheorem{proposition}{Proposition}

\theoremstyle{remark}
\newtheorem{remark}{Remark}

\theoremstyle{definition}

\numberwithin{equation}{section}

\newcommand*{\pk}{\mathscr{H}}

\newcommand{\bpm}{\begin{pmatrix}}
\newcommand{\epm}{\end{pmatrix}}
\newcommand{\bsm}{\lt(\begin{smallmatrix}}
\newcommand{\esm}{\end{smallmatrix}\rt)}
\newcommand{\lt}{\left}
\newcommand{\rt}{\right}

\begin{document}

\vskip 5mm

\title[Sign changes of Fourier coefficients of modular forms]
{Sign changes of Fourier coefficients of modular forms of half integral weight, 2}
\author{Y.-J. Jiang, Y.-K. Lau,  G.-S. L\"u, E. Royer \& J. Wu}

\address{
Yujiao Jiang
\\
Department of Mathematics
\\
Shandong University
\\
Jinan, Shandong 250100
\\
China}
\email{yujiaoj@hotmail.com}

\address{
Yuk-Kam Lau
\\
Department of Mathematics
\\
The University of Hong Kong
\\
Pokfulam Road
\\
Hong Kong}
\email{yklau@maths.hku.hk}

\address{
Guangshi L\"u
\\
Department of Mathematics
\\
Shandong University
\\
Jinan, Shandong 250100
\\
China}
\email{gslv@sdu.edu.cn}

\address{%
Emmanuel Royer\\
Clermont Universit\'e\\
Universit\'e Blaise Pascal\\
Laboratoire de math\'ematiques\\
BP 10448\\
F-63000 Clermont-Ferrand\\
France %
}
\curraddr{%
Emmanuel Royer\\
Universit\'e Blaise Pascal\\
Laboratoire de math\'ematiques\\
Les C\'ezeaux\\
BP 80026\\
F-63171 Aubi\`ere Cedex\\
France %
}
\email{{emmanuel.royer@math.univ-bpclermont.fr}}

\address{%
Jie Wu\\
CNRS\\
Institut \'Elie Cartan de Lorraine\\
UMR 7502\\
F-54506 Van\-d\oe uvre-l\`es-Nancy\\
France}
\curraddr{%
Université de Lorraine\\
Institut \'Elie Cartan de Lorraine\\
UMR 7502\\
F-54506 Van\-d\oe uvre-l\`es-Nancy\\
France
}
\email{jie.wu@univ-lorraine.fr}

\date{\today}

\begin{abstract}
In this paper, we investigate the sign changes of Fourier coefficients of half-integral weight Hecke eigenforms
and give two quantitative results on the number of sign changes.
\end{abstract}

\subjclass[2000]{11F30}
\keywords{Fourier coefficients, half-integral weight modular forms, sign-changes, truncated Voronoi series}
\maketitle

\addtocounter{footnote}{1}

\section{Introduction}

The study of sign-changes of Fourier coefficients of automorphic forms is recently very active. For modular (Hecke eigen-)forms of integral weight, the consequential result from Matom\"aki and Radziwill \cite{MR}  is exceptionally charming, where the multiplicative properties of the Fourier coefficients play a substantial role. However the modular forms of half-integral weight do not share the same kind of multiplicativity,  and many problems deserve delving. 

Let \(\ell\geq 2\)  be a positive integer, and denote by \(\cush\) the set of all cusp forms of weight \(\ell+1/2\) for the congruence subgroup \(\Gamma_0(4)\).  Consider the coefficients in the Fourier expansion of a complete Hecke eigenform \(\f\in\cush\) at \(\infty\),
\begin{equation}\label{eq_newone}%
\f(z)=\sum_{n\ge 1} \lambda_{\f}(n) n^{\ell/2-1/4}\ee(nz) \quad (z\in\pk),
\end{equation}
where \(\ee(z) = \ee^{2\pi\ic z}\) and \(\pk\) is the Poincar\'e upper half plane. A specific question is the number of sign-changes  when all $\lambda_\f(n)$ are real. We interlude with the meaning of sign-changes of a sequence. 

Let $\mathcal{N}$ be a subset of $\N$ endowed with the ordering of integers. The sets of squarefree integers or arithmetic progressions  are basic examples. Given a real sequence  $\{a_n\}_{n\in \mathcal{N}}$. A sign-change is realized via a closed and bounded interval $[i,j]\subset (0,\infty)$ such that 
\begin{itemize}
\item[(i)]
its end-points $i,j$ lie in $\mathcal{N}$ and satisfy $a_ia_j<0$, and 
\item[(ii)]
$a_n=0$ for all $n\in (i,j)\cap \mathcal{N}$. 
\end{itemize}
The sequence $\{a_n\}_{n\in \mathcal{N}}$ is said to have a sign-change in the interval $I$  if $I$ contains one such interval $[i,j]$.  Besides,  the number of sign-changes of $\{a_n\}_{n\in \mathcal{N}}$ in $[1,x]$, denoted by $\cC^\mathcal{N}(x)$, is meant to be the number of intervals $[i,j]$ contained in $[1,x]$.\footnote{An equivalent but slightly  different formulation is given in \cite{LauRoyerWu2014}.}

Let $\flat$ be the set of squarefree numbers. Hulse, Kiral, Kuan \& Lim \cite{HKKL2012} proved 
that the sequence $\{\lambda_{\f}(t)\}_{t\in\flat}$ has an infinity of sign-changes.  
A quantitative version is given in Lau, Royer \& Wu \cite[Theorem 4]{LauRoyerWu2014}, 
which says $\cC_\f^\flat(x) \gg  x^{(1-4\varrho)/5-\varepsilon}$ 
where  $\cC_\f^\flat(x)$ denotes the number of sign-changes of $\{\lambda_{\f}(t)\}_{n\in \flat}$  in $[1,x]$
and the constant $\varrho$ is determined by \eqref{varrho} below. Conjecturally $\varrho=\varepsilon$ but it is still hard to guess the tight lower bound. 
On the other hand, Meher \& Murty \cite{MeherMurty2014} studied the sign-change problem 
for Hecke eigenforms $\f$ in Kohnen plus subspace of \(\cush\). 
A form $\f$ in the plus space has its Fourier coefficients  supported at integers $n\equiv 0$ or $(-1)^\ell\,(\bmod\,{4})$, i.e.  
$\f$ has the Fourier expansion at $\infty$ of the form
$$
\mathfrak{f}(z)
= \sum_{(-1)^{\ell} n \equiv 0, 1 ({\rm mod}\,4)} \lambda_{\f}(n) n^{\ell/2-1/4} {\rm e}^{2\pi {\rm i}n z}.
$$
When $\f$ is a Hecke eigenform  in the plus space and its coefficients $\lambda_f(n)$ are all real, 
Meher \& Murty proved in \cite[Theorem 2]{MeherMurty2014} that $\{\lambda_f(n)\}_{n\in \N}$ has a sign-change 
in the short interval $(x, x+x^{43/70+\varepsilon}]$ for any $\varepsilon>0$ and for all sufficiently large $x\ge x_0(\vep)$. 
An immediate consequence is $\cC_\f^\N (x) \gg x^{27/70-\vep}$. This work naturally motivates the sign-change problem for arithmetic progressions. 

In this paper, we furnish progress, based on our work in \cite{JLLRW}, in the above problems for complete Hecke eigenforms $\f \in\cush$. Firstly for the case $\mathcal{N}=\flat$, we sharpen the lower bound for $\cC_\f^\flat(x)$. 
\begin{theorem}\label{thm1}
Let  \(\ell\geq 2\)  be an integer and  \(\f\in\cush\) a complete Hecke eigenform such that its Fourier coefficients are real. Let $\varrho$ be defined as in \eqref{varrho} below, and  $\vartheta$ any number satisfying
$$
0<\vartheta < \min(\tfrac{1-2\varrho}3, \tfrac14).
$$
Then
\begin{equation}\label{LB:thm1}
\cC_{\f}^\flat (x)\gg_{\f, \vartheta} x^{\vartheta}
\end{equation}
for all \(x\geq x_0(\f,\vartheta)\), where the constant \(x_0(\f, \vartheta)\)  and the implied constant depend on \(\f\) and \(\vartheta\) only.
\end{theorem}

\begin{remark}
In particular, Conrey \& Iwaniec \cite{ConreyIwaniec2000} gives $\varrho=\frac16+\vep$ which leads to
\[
\cC_{\f}^\flat (x)\gg_{\f, \varepsilon} x^{2/9-\varepsilon}
\]
for all \(x\geq x_0(\f, \varepsilon)\), improving the exponent $\frac1{15}-\varepsilon$ in \cite{LauRoyerWu2014}.
\end{remark}

Secondly we generalize the case of $\mathcal{N}=\N$ in Meher \& Murty  \cite{MeherMurty2014} to arithmetic progressions. Let $Q\ge 1$ be an integer, and $a= 0$ or $a\in \N$ with $(a,Q)=1$. Define
\begin{equation}\label{Adef}
\mathcal{A} = \mathcal{A}_{a,Q}:= \{n\in \N:\ n\equiv a\,(\bmod\,{Q})\}.
\end{equation}
We  study the sign-changes of $\{\lambda_\f(n)\}_{n\in \mathcal{A}}$ and sharpen the exponent $\tfrac{43}{70}+\vep$ of Meher \& Murty's result to $\tfrac{1}{2}$, which in turn gives the better lower bound $\cC_\f^\N (x) \gg x^{1/2}$.  
\begin{theorem}\label{thm2}
Assume  the same conditions for \(\f\) and $\varrho$ in Theorem~\ref{thm1}. Let $Q\ge 1$ be odd and $\mathcal{A}=\mathcal{A}_{a,Q}$ defined as in \eqref{Adef}. Suppose one of the following condition holds:
\begin{itemize}
\item[$1^\circ$]
$Q=1$;
\item[$2^\circ$]
$a=0$ and $Q=\prod_{p\mid Q} p^{\alpha_p}$ where all $\alpha_p$ are odd;
\item[$3^\circ$]
$(a,Q)=1$ and $Q=\prod_{p\mid Q} p^{\alpha_p}$ where all $\alpha_p$ are $\ge 2$.
\end{itemize}
Then there are positive constants $c_0=c_0(\mathfrak{f},Q)$ and $x_0= x_0(\f,Q)$ such that 
the sequence $\{\lambda_{\f}(n)\}_{n\in \mathcal{A}}$ has at least one sign change in the interval $(x, x+c_0 x^{1/2}]$ for all $x\ge x_0$. In particular, we have
$$
\cC_{\f}^{\mathcal{A}}(x)\gg_{\mathfrak{f},Q} x^{1/2}
$$
for all $x\ge x_0$.
\end{theorem}

\vskip 8mm

\section{Methodologies}

Let $\lambda_\f(n)$  be the coefficients as in \eqref{eq_newone} and $\mathcal{N}$  a subset of $\N$. Define
\begin{equation}\label{def_FirstMoment}
S_{\f}^\mathcal{N}(x) 
:= \sum_{\substack{n\le x\\ n\in \mathcal{N}}} \lambda_{\f}(n).
\end{equation}
A typical approach for the sign-change detection exploits the oscillation exhibited in  the mean $S_{\f}^\mathcal{N}(x)$, while to locate the sign-change, the mean over short intervals, i.e. $S_{\f}^\mathcal{N}(x+h)- S_{\f}^\mathcal{N}(x)$ for small $h$, will be a good device. Suppose a sign-change is found in the interval $[x,x+h]$ for every $x$ large enough. Then it follows immediately that  the number of sign-changes in $[1,x]$ is at least $x/h + O(1)$ (and hence $\gg x/h$). 
A standard way to study $S_{\f}^\mathcal{N}(x)$ is via the Dirichlet series. But for various $\mathcal{N}$, we get different degree of its analytic information. 

For $\mathcal{N}=\flat$, i.e. the case of squarefree integers, we  only get an analytic continuation of the Dirichlet series
\begin{equation}\label{def:Lfbs}
L_\f^\flat(s) := \sideset{}{^\flat}\sum_{t\ge 1} \lambda_{\f}(t) t^{-s}
\end{equation} 
in the half-plane $\re s>\tfrac{1}{2}$, where $\sum_{t\ge 1}^\flat$ ranges over squarefree integers $t\ge 1$. 
As illustrated in \cite{LauRoyerWu2014}, it turns out that  the weighted mean is more effective.  Thus, to prove Theorem~\ref{thm1}, we first derive \eqref{mean} below, 
\begin{align}\label{mean}
\sideset{}{^\flat}\sum_{x\leq t\leq x+h} \lambda_{\f}(t)
\min\left\{\log\bigg(\frac{x+h}{t}\bigg), \, \log\bigg(\frac{x}{t}\bigg)\right\}
\ll_{\varepsilon} h^\frac{1}{2} x^{\varepsilon}.
\end{align}
The better exponent $\tfrac{1}{2}$ (versus $\tfrac{3}{4}$ in \cite{LauRoyerWu2014}) of $h$ is a key for the improvement.  Another key is to have a mean square formula with better $O$-term. In \cite{LauRoyerWu2014}, we showed that 
$$
\sum_{X<n\le 2X} |\lambda_{\f}(n)|^2 
= D_{\f}X+O_{\f,\vep}\big(X^{\beta+\vep}\big).
$$
with $\beta=\tfrac{3}{4}+\varrho$. Here we sharpen it to $\beta=\tfrac{3}{4}$ in Lemma~\ref{lem1} 
and then conclude Theorem~\ref{thm1} with argument in \cite{LauRoyerWu2014}. 
This will be done in Section~\ref{Sthm1}.

Next for  $\mathcal{N}=\mathcal{A}$ (see \eqref{Adef}), 
we shall provide a truncated Voronoi formula for $S_\f^\mathcal{A}(x)$ in Section~\ref{vor}. 
This result is itself interesting since the Voronoi formula is an vital tool for many applications, 
see \cite{ivic}, \cite{jutila} for example. 
Then we complete the proof of Theorem~\ref{thm2} with the method of Heath-Brown and Tsang \cite{HBT94}. 
However the congruence condition underlying $\mathcal{A}$ gives rise to new (but interesting) difficulties. To transform the congruence, additive characters of modulus $d|Q$ will be invoked and then two consequences follow:  the  summands in the Voronoi formula are intertwined with Kloosterman-Sali\'e sums, and the frequencies in the cosines are of the form $\sqrt{n}/d$. We need to select a suitable frequency for amplification with a pair of non-vanishing Sali\'e sum and Fourier coefficient in the associated summand. The implementation is successful when $Q$ fulfills the conditions in Theorem~\ref{thm2}, which will be elucidated in Sections~\ref{pthm2} \& \ref{Sthm2}.
It is worthwhile to remark that the mean square result of $\lambda_\f(n)$ is not needed for the method in \cite{HBT94}.

\vskip 8mm

\section{Background}

A cusp form $\f\in \cush$ has Fourier expansions at the three inequivalent cusps $\infty, -\tfrac{1}{2}, 0$ of 
$\Gamma_0(4)$, 
which are respectively given by \eqref{eq_newone}, and \eqref{gfe}, \eqref{hfe} below:
\begin{equation}\label{gfe}
\begin{aligned}
\g(z) 
& := 2^{\ell+1/2} (-8z+1)^{-(\ell+1/2)} \f\bigg(\frac{4z}{-8z+1}\bigg)
\\
& \,= 2^{\ell+1/2} \sum_{n\ge 1} \lambda_\g(n) n^{\ell/2-1/4} \ee(nz)
\end{aligned}
\end{equation}
and 
\begin{equation}\label{hfe}
\fh(z)
:=  (-\ic 2z)^{-(\ell+1/2)} \f\bigg(\frac{-1}{4z}\bigg) =  \sum_{n\ge 1} \lambda_\fh(n) n^{\ell/2-1/4} \ee(nz).
\end{equation}
Following the argument in \cite[Section 2.2]{LauRoyerWu2014}, we have
\begin{equation}\label{msq}
\sum_{n\le x} |\lambda_f(n)|^2 \sim x
\qquad
(\text{for all three cases $f=\f,\g,\fh$}).
\end{equation}

When ${\f}$ is a complete Hecke eigenform, we know from \cite{JLLRW} that $\g$ 
and $\fh$ are Hecke eigenforms of $\mathsf{T}(p^2)$ for all odd prime $p$.  
A consequence is, cf. \cite[Lemma 3.2 with $\mathcal{Q}=\{2\}$]{JLLRW}: for all odd $m\ge 1$, all squarefree $t$ and $j\ge 0$, 
\begin{equation}\label{mul}
\lambda_f(2^jt)=0 \ \Rightarrow \ \lambda_f(2^jtm^2)=0\quad \mbox{($f=\f,\g,\fh$)}.
\end{equation}
In addition, we have the following pointwise estimate, see  \cite[Lemma 3.3]{JLLRW}.

\begin{lemma}\label{lem3.1} 
Let $\f$ be a complete Hecke eigenform, $\g$ and $\fh$ be defined as above. 
For any integer $m=tr^2$ where $t\ge 1$ is squarefree, we have
\begin{eqnarray*}
\lambda_f(m)\ll_\f |\lambda_f(t)|\tau(r)^2 +|\lambda_\f(t)| \tau(r)^2\ll_{\f,\varrho} t^\varrho \tau(r)^2
\end{eqnarray*}
for $f=\f,\g,\fh$ respectively, where $\tau(n)$ is the divisor function and $\varrho$ satisfies \eqref{varrho} below. 
The first implied $\ll$-constant depends only $\f$ and the second implied $\ll$-constant depends 
at most on $\f$ and $\varrho$.
\end{lemma}

Here $\varrho$ denotes the exponent for which 
\begin{equation}\label{varrho}
\lambda_\f (t) \ll_\varrho t^{\varrho}\qquad \mbox{ $\forall$ $t$ squarefree},
\end{equation}
i.e. the bound towards the Ramaujan Conjecture for the half-integral weight Hecke eigenforms. 
The conjectural value is $\varrho=\vep$. 
Conrey \& Iwaniec \cite{ConreyIwaniec2000} obtained $\varrho=\frac16+\vep$.  

Let $d\ge 1$ be an integer and $(u,d)=1$. Define the twisted $L$-function for $\f$ by 
\begin{equation}\label{eq3.6}
L_\f(s, u/d) = \sum_{m\geq 1} \frac{\lambda_{\f}(m) \ee(m u/d)}{m^s}
\qquad \mbox{ $(\re s>1)$}
\end{equation}
and define similarly for $\g$ and $\fh$. 
These twisted $L$-functions when attached with suitable factors may be expressed as integrals of $\f$ along vertical geodesics, and extend to entire functions, cf. \cite[(4.4)-(4.5)]{HKKL2012}.  Moreover Hulse et al  found the  functional equation for $L_\f(s,u/d)$, which is put in the following form 
\begin{equation}\label{fe}
q_d^s L_\infty (s) L_\f(s, u/d) = \ic^{-(\ell+1/2)} q_d^{1-s} L_\infty(1-s) \widetilde{L}_\f (1-s, v/d),
\end{equation}
where  $uv\equiv 1\,(\bmod\,{d})$ and 
$L_\infty (s) := (2\pi)^{-s} \Gamma\big(s+\tfrac{\ell}2-\tfrac14)$ is the gamma factor, cf. \cite[Lemma 4.3]{HKKL2012}  
and \cite{JLLRW}. 
The conductor $q_d$ and the dual $L$-function $\widetilde{L}_\f(s,v/d)$ are defined as follows:
\begin{eqnarray}\label{qd}
q_d=d \ \mbox{  or } \ 2d\ \mbox{ according to $4\mid d$ or not,}
\end{eqnarray}
and
\begin{equation}\label{Ltilde}
\widetilde{L}_\f(s,v/d) := \sum_{n\ge 1} \lambda(n; d) \varpi_d(n,v) n^{-s},
\end{equation}
where 
\begin{equation}\label{lambdafd}
{\renewcommand{\arraystretch}{1.8}
\renewcommand{\tabcolsep}{2.5mm}
\begin{tabular}{|c|c|c|c|}
\hline
{}
& $\lambda(n; d)$ 
& $\varpi_d(n,v)$
\\
\hline
$4\mid d$
& $\lambda_{\f}(n)$  
& $\vep_v^{2\ell+1} \big(\tfrac{d}v\big) \ee\big(\tfrac{-nv}{d}\big)$  
\\
\hline
$2\,\|\,d$
& $\lambda_{\g}(n)$  
& $\vep_v^{2\ell+1} \big(\tfrac{d}v\big) \ee\big(\tfrac{-nv}{4d}\big)$
\\
\hline
$2\nmid d$
& $\lambda_{\fh}(n)$  
& $\ic^{\ell+1/2} \vep_d^{-(2\ell+1)} \big(\tfrac{v}d\big) \ee\big(\tfrac{-\overline{4}nv}{d}\big)$  
\\
\hline
\end{tabular}
}
\end{equation}
with $4\overline{4}\equiv 1\,(\bmod\,{d})$. 

In \cite{HKKL2012}, Hulse et al applied $L_\f(s,u/d)$ to obtain the analytic properties of $L_\f^\flat(s)$, 
which was sharpened to the following result \cite[Theorem 1]{JLLRW}. 
\begin{lemma}\label{main}
For a complete Hecke eigenform ${\mathfrak f}\in \cush$, the series $L_\f^\flat(s)$ extends analytically to a holomorphic function on $\Re e\, s>\tfrac{1}{2}$, and for any $\vep>0$,
\begin{equation}\label{UBMs}
L_\f^\flat(s)
\ll_{{\mathfrak f}, \varepsilon} (|\tau|+1)^{1-\sigma+2\varepsilon}
\qquad
(\tfrac1{2}+\varepsilon\le\sigma\le 1+\vep, \tau\in \R),
\end{equation}
where the implied constant depends on ${\mathfrak f}$ and $\varepsilon$ only.
\end{lemma}
\begin{remark} Using  Lemma~\ref{main} in place of \cite[Proposition 7]{LauRoyerWu2014}, the estimate in \eqref{mean} follows plainly from the same argument as in \cite[Section 4.1]{LauRoyerWu2014}, so we do not repeat here. 
\end{remark}

\vskip 8mm

\section{Proof of Theorem \ref{thm1}}\label{Sthm1}

We start with the following lemma where the $O$-term in \eqref{2nd_moment} is smaller than \cite[(14)]{LauRoyerWu2014}.
\begin{lemma}\label{lem1}
Let \(\ell\geq 2\)  be a positive integer and \(\f\in\cush\) be  a complete Hecke eigenform.
Then for any $\varepsilon>0$ and all $x\ge 2$, we have
\begin{equation}\label{2nd_moment}
\sum_{n\le x} |\lambda_{\f}(n)|^2
= D_{\f} \, x + O_{\f, \varepsilon}\big(x^{3/4+\varepsilon}\big),
\end{equation}
where $D_{\f}$ is a positive constant depending on $\f$.
\end{lemma}

\begin{proof}
We choose two smooth compactly supported functions $w_{\pm}$ such that
\begin{itemize}
\item
$w_{-}(x)=1$ for $x\in[X+ Y,2X-Y], w_{-}(x)=0$ for $x\geq 2X$ and $x\leq X$; 
\item
$w_{+}(x)=1$ for $x\in[X,2X], w_{+}(x)=0$ for $x\geq 2X+Y$ and $x\leq X-Y$;
\item
$w_{\pm}^{(j)}(x)\ll_j Y^{-j}$ for all $j\geq 0$;
\item
the Mellin transform of $w(x)$ is
\begin{equation}\label{UB_Mw}
\begin{aligned}
\widehat{w_{\pm}}(s)
& := \int_0^\infty w_{\pm}(x)x^{s-1} \d x
\\
& = \frac{1}{s\cdots (s+j-1)}\int_0^\infty w_{\pm}^{(j)}(x)x^{s+j-1} \d x
\\
& \ll_j \frac{Y}{X^{1-\sigma}} \left(\frac{X}{|s|Y}\right)^j\quad \mbox{$\forall$ $j\ge 1$};
\end{aligned}
\end{equation}
\item 
trivially $\widehat{w_{\pm}}(s)\ll X^\sigma$ and
\begin{equation}\label{X+Y}
\widehat{w_{\pm}}(1)=X+O(Y).
\end{equation}
\end{itemize}

Obviously we have
\begin{equation}\label{smooth}
\sum_{n} |\lambda_{\f}(n)|^2 w_{-}(n)
\leq \sum_{X<n\le 2X} |\lambda_{\f}(n)|^2 
\leq \sum_{n} |\lambda_{\f}(n)|^2 w_{+}(n).
\end{equation}
Let the Dirichlet series associated with $|\lambda_{\f}(n)|^2$ be defined as (see e.g. \cite[(11)]{LauRoyerWu2014})
\begin{equation*}
D(\f\otimes \overline{\f}, s) = \sum_{n=1}^\infty |\lambda_{\f}(n)|^2 n^{-s}.
\end{equation*}
By the Mellin inversion formula
$$
w_{\pm}(x) = \frac{1}{2\pi\ic} \int_{2-\ic\infty}^{2+\ic\infty} \widehat{w_{\pm}}(s) x^{-s} \d s,
$$
we write
$$
\sum_{n} |\lambda_{\f}(n)|^2 w_{\pm}(n)
= \frac{1}{2\pi \ic}\int_{(2)} \widehat{w_{\pm}}(s) D(\f\otimes \overline{\f}, s) \d s.
$$
With the help of Cauchy's residue theorem, we obtain that
\begin{equation}
\sum_{n} \lambda_{\f}(n)^2w_{\pm}(n)
= D_{\f}\widehat{w_{\pm}}(1) +\frac{1}{2\pi \ic}\int_{(\kappa)}\widehat{w_{\pm}}(s) D(\f\otimes \overline{\f}, s) \d s,
\end{equation}
where $\tfrac{1}{2}<\kappa <1$ and $D_{\f} := \text{Res}_{s=1}  D(\f\otimes \overline{\f}, s)$.
By \eqref{X+Y}, \eqref{UB_Mw} with $j=2$ and 
the convexity bound \cite[Proposition 7]{LauRoyerWu2014}
$$
D(\f\otimes \overline{\f}, s)\ll _{\f,\vep }(1+|\tau|)^{2\max(1-\sigma,0)+\vep}
\qquad
(\tfrac{1}{2}<\sigma \leq 3),
$$  we derive
$$
\sum_{n} |\lambda_{\f}(n)|^2 w_{\pm}(n) 
= D_{\f}X+O_{\f,\vep}\big(Y + X^{1+\kappa} Y^{-1}\big).
$$
Taking $\kappa = \tfrac{1}{2}+\vep$ and $Y=X^{3/4}$, and combining the obtained estimation with \eqref{smooth}, we find that
$$
\sum_{X<n\le 2X} |\lambda_{\f}(n)|^2 
= D_{\f}X+O_{\f,\vep}\big(X^{3/4+\varepsilon}\big),
$$
which implies \eqref{2nd_moment} after a dyadic summation.
\end{proof}

Now we return to prove the theorem. Take $h=x^\eta$ where $\eta >\tfrac{3}{4}$ is specified later.  
Lemma~\ref{lem1} gives
\begin{equation*}
{\rm (i)} \quad  Ch \leq \sum_{x\leq n\leq x+h} \lambda_{\f}(n)^2 
\quad \mbox{ and }\quad 
{\rm (ii)} \quad \sum_{x/m^2\leq t\leq (x+h)/m^2} \lambda_{\f}(n)^2 \ll hm^{-3/2}
\end{equation*}
for any $m\le \sqrt{x+h}$, 
where the positive constant $C$ and the implied $\ll$-constant depend on $\f$ and $\eta$ only. 
Combining (i) with Lemma~\ref{lem3.1} leads to
\begin{equation*}
Ch 
\leq \sum_{x\leq n\leq x+h} \lambda_{\f}(n)^2 
\leq C'\sum_{m\leq \sqrt{x+h}}\tau(m)^4 
\sideset{}{^\flat}\sum_{x/m^2\leq t\leq (x+h)/m^2} \lambda_{\f}(t)^2
\end{equation*}
where $\sum^\flat$ confines the running index over squarefree integers only and 
$C'>0$ is a constant depending at most on $\f$. 
By (ii) and the fact $\sum_{m\ge A} \tau(m)^4 m^{-3/2} \gg A^{-1/2+\vep}$, 
we conclude that for a large enough constant $A$, 
\begin{equation*}
\sum_{m\leq A}\tau(m)^4 \sumb_{x/m^2\leq t\leq (x+h)/m^2} \lambda_{\f}(t)^2 
\ge \{C/C' + O(A^{-1/2+\vep})\} h 
\gg h
\end{equation*}
which is \cite[(23)]{LauRoyerWu2014}. 
Thus, repeating the same argument (in \cite[(24)-(26)]{LauRoyerWu2014}),  
we obtain \cite[(26)]{LauRoyerWu2014} with a smaller admissible $h=x^\eta$ 
(here $\eta>\tfrac{3}{4}$ is required instead of $\eta>\tfrac{3}{4}+\varrho$).

Next we note that the new estimate \eqref{mean} improves the upper bound $h^{3/4} x^\vep$ in \cite[(21) of Section 4.2]{LauRoyerWu2014} to $h^{1/2} x^\vep$. Consequently, we get the new lower bound  
$$
x^{-1-\varrho}h^2 +O(h^{1/2}x^\vep)
$$
for \cite[(27)]{LauRoyerWu2014}. The optimal choice of $\eta$ is $\tfrac{2}{3}(1+\varrho)+\vep$, 
and together with the constraint $\eta>\tfrac{3}{4}$, we  choose
$$
\eta = \max\big\{\tfrac{2}{3}(1+\varrho), \tfrac34\big\}+\vep.
$$
We complete the proof of Theorem~\ref{thm1} with the same argument in remaining part of \cite[Section 4.2]{LauRoyerWu2014}.

\vskip 8mm

\section{Preparation for the truncated Voronoi formula}

Applying the additive character to replace the congruence condition, that is,
$$
Q^{-1} \sum_{d\mid Q} \sumstar_{u\,({\rm mod}\,d)} \ee\bigg(\frac{u(n-a)}d\bigg)
= \delta_{n\equiv a\, ({\rm mod}\, Q)} 
$$
where $\delta_*=1$ if $*$ holds and $0$ otherwise, we have
\begin{equation}\label{asum}
\mathcal{S}_\f^\mathcal{A}(x)
:= \sum_{\substack{n\le x\\ n\equiv a ({\rm mod}\,Q)}} \lambda_{\f}(n)
= Q^{-1} \sum_{d\mid Q} \mathcal{S}_\f(x, a/d),
\end{equation} 
where 
\begin{equation}\label{def_FirstMoment}
\mathcal{S}_{\f}(x,a/d) := \sideset{}{^*}\sum_{u\,({\rm mod}\, d)} \ee\bigg(\frac{-au}d\bigg) 
\sum_{n\le x} \lambda_{\f}(n) \ee\bigg(\frac{nu}d\bigg).
\end{equation}
Here $\sumstar_{\hskip -1,5mm u ({\rm mod}\, d)}$ denotes the sum over $u\,(\bmod\,{d})$ with $(u,d)=1$. 
The inner sum over $n$ is clearly associated with $L_\f(s,u/d)$, thus we  introduce the auxiliary function
\begin{eqnarray}\label{calL}
\mathcal{L}_\f(s,a/d):= \sideset{}{^*}\sum_{u\, ({\rm mod}\, d)} \ee\!\left(-\frac{au}d\right)  L_\f(s,{u}/d).
\end{eqnarray}
The Dirichlet series associated to $\mathcal{S}^\mathcal{A}(x)$,
\begin{equation}\label{defLsf}
L_\f(s,a,Q) := \sum_{\substack{n\ge 1\\ n\equiv a ({\rm mod} \, Q)}} \lambda_{\f}(n) n^{-s}
\end{equation}
is equal to  
\begin{equation}\label{LL}
L_\f(s,a,Q) = Q^{-1} \sum_{d\mid Q} \mathcal{L}_\f(s,a/d).
\end{equation}

Plainly $\mathcal{L}_\f(s,a/d)$ satisfies a functional equation by \eqref{fe},
\begin{eqnarray}\label{auxLfe}
q_d^{s} L_\infty(s) \mathcal{L}_\f(s,a/d) = \ic^{-(\ell+1/2)} q_d^{1-s} L_\infty(1-s)  \widetilde{\mathcal{L}}_\f(1-s,a/d) 
\end{eqnarray}
where $\widetilde{L}_\f (s, v/d)$ is defined as in \eqref{Ltilde} and
\begin{equation*}
\widetilde{\mathcal{L}}_\f(s,a/d) = \sumstar_{u\, ({\rm mod}\, d)} \ee\left(-\frac{au}d\right) 
\widetilde{L}_\f (s, \overline{u}/d)\qquad 
(u\overline{u}\equiv 1\,(\bmod\,{d})).
\end{equation*}
When  $\re s > 1$, we may express $\widetilde{\mathcal{L}}_\f(s,a/d)$ as a Dirichlet series whose coefficients are products of $\lambda(n; d)$ and the Kloosterman-Sali\'e sums. Indeed,  by \eqref{Ltilde}, we have
\begin{eqnarray}\label{auxLs}
\widetilde{\mathcal{L}}_\f(s,a/d) =  \sum_{n\ge 1} \lambda(n; d) {\rm K}(a, n; d) n^{-s} 
\end{eqnarray}
where (noting $v=\overline{u}\,(\bmod\,{d})$),
\begin{equation}\label{def:Kand}
{\rm K}(a, n; d) 
:=  \sideset{}{^*}\sum_{u\, ({\rm mod}\, d)} \varpi_d(n,\overline{u}) \ee\!\left(-\frac{au}d\right).
\end{equation}
By \eqref{lambdafd}, 
$$
{\rm K}(a, n; d)
= \left\{\begin{array}{ll}
\displaystyle  \sumstar_{u\, ({\rm mod}\, d)} \vep_u^{2\ell+1} 
\left(\frac{d}{u}\right) \ee\!\left(-\frac{a\overline{u}+ nu}{4d}\right) 
& \mbox{ if $4\mid d$}, 
\\\noalign{\vskip 0,1mm}
\displaystyle  \sumstar_{u\, ({\rm mod}\, d)} \vep_u^{2\ell+1} 
\left(\frac{d}{u}\right) \ee\!\left(-\frac{4a\overline{u}+ nu}{4d}\right)
&  \mbox{ if $2\,\|\,d$}, 
\\\noalign{\vskip 0,1mm}
\ic^{\ell+1/2} \vep_d^{-(2\ell+1)} 
\displaystyle  \sumstar_{u\, ({\rm mod}\, d)} 
\bigg(\frac{u}{d}\bigg) \ee\!\left(-\frac{a\overline{u}+\overline{4} nu}d\right)  
& \mbox{ if $2\nmid d$}. 
\end{array}
\right.
$$   

\begin{lemma} 
Let $\tau(d)$ be the divisor function. 
We have 
\begin{equation}\label{Kbound}
|{\rm K}(a, n; d)|\ll (d,n)^{1/2} d^{1/2} \tau(d).
\end{equation}
Moreover, for the case  $2\nmid d$,  if there exists $x\in \{a,n\}$ such that $(x,d)=1$, then
\begin{eqnarray}\label{Kform}
{\rm K}(a, n; d) = \ic^{\ell+1/2} \vep_d^{-2\ell} d^{1/2} \bigg(\frac{x}d\bigg) 
\sum_{y^2\equiv an ({\rm mod}\, d)} \ee\bigg(\frac{y}d\bigg).
\end{eqnarray}
\end{lemma}

\begin{proof}
We express ${\rm K}(a, n; d)$ in terms of Kloosterman-Sali\'e sums (see Appendix for their definitions), as follows:
\begin{equation}\label{KS}
{\rm K}(a, n; d)
= \left\{
\begin{array}{ll}
\overline{K_{2\ell+1}(n,a;d)} & \mbox{ for $\;4\mid d$},\vspace{1mm}
\\\noalign{\vskip 0,5mm}
\tfrac14 \overline{K_{2\ell+1}(n,a; 4d)} & \mbox{ for $\;2\,\| \,d$}, \vspace{1mm}
\\\noalign{\vskip 0,5mm}
 \ic^{\ell+1/2} \vep_d^{-(2\ell+1)}\, \overline{S (\overline{4}n, a; d)} & \mbox{ for $\;2\nmid d$},
\end{array}
\right.
\end{equation}
where in the case of $2\,\|\,d$, the range of summation is enlarged  to a reduced residue system $(\bmod\,{4d})$. 
From \eqref{eq:9.1} below, we have
\begin{equation}\label{Kbound}
|{\rm K}(a, n; d)|\ll (d,n)^{1/2} d^{1/2} \tau(d).
\end{equation}

The formula \eqref{Kform} follows from the result in \cite[Lemma 4.9]{Iwaniec1997} for the Sali\'e sum.
\end{proof} 

\begin{lemma}\label{convex}  
Let  $d\ge 1$ and $a$ be any integers.  
For any $\vep>0$, we have  
\begin{equation}\label{Lb1}
\mathcal{L}_\f(\sigma +\ic \tau,a/d) 
\ll d^{(3-\sigma)/2+2\vep} (1+|\tau|)^{1-\sigma+2\vep}
\quad (-\vep\le \sigma \le 1+\vep, \, \tau\in\R),
\end{equation}
where the implied $\ll$-constant depends on $\f$ and $\vep$ only. 
\end{lemma}
\begin{proof} Let $\re s=1+\vep$. By \eqref{msq} and \eqref{eq3.6}, we have trivially $L_\f(s,u/d)\ll_\vep  1$ and with \eqref{calL}, $\mathcal{L}_\f(s,a/d)\ll_\vep  d$. 
Next for $\re s=-\vep$,   we infer from \eqref{auxLfe} and \eqref{auxLs} that 
$$
\mathcal{L}_\f(s,a/d)
= \ic^{-(\ell+1/2)} q_d^{1-2s} \frac{L_\infty(1-s)}{L_\infty(s)}
\sum_{n\ge 1} \frac{\lambda(n; d) {\rm K}(a, n; d)}{n^{1-s}}.
$$
Thus, with \eqref{Kbound} and Stirling's formula, it follows that
\begin{align*}
\mathcal{L}_\f(-\vep +\ic \tau,a/d)
& \ll (d^{3/2}(1+|\tau|))^{1+\vep}
\sum_{n\ge 1} |\lambda(n; d)|(n,d)^{1/2} n^{-(1+\vep)}
\\
& \ll (d^{3/2}(1+|\tau|))^{1+\vep}
\end{align*}
because $|\lambda(n; d)|(n,d)^{1/2}\le |\lambda(n; d)|^2 +(n,d)$, implying that the last summation is 
\begin{eqnarray*}
\ll \sum_{n\ge 1} |\lambda(n; d)|^2n^{-(1+\vep)} + \sum_{l\mid d} l^{-\vep} \sum_{n\ge 1}  n^{-(1+\vep)}\ll \tau(d).
\end{eqnarray*}
An application of Phragm\'en–Lindel\"of principle completes the proof. 
\end{proof}

\goodbreak
\vskip 8mm

\section{Truncated Voronoi formula}\label{vor}

This section is devoted to the Voronoi formulas. In order for a simpler form for the result, 
let us set, with the notation \eqref{def:Kand},
\begin{eqnarray}\label{eq7.3}
\phi_a(n,d) 
:= \sqrt{q_d}\, \ic^{-(\ell+1/2)}{\rm K}(a, n; d) \ll   (n,d)^{1/2} \tau(d)d
\end{eqnarray} 
by \eqref{Kbound}, and trivially $|\phi_a(n,d)|\le \sqrt{2} d^{3/2}$. We have the following result.

\begin{theorem}\label{Voronoi}
Let $\ell\ge 2$ be an integer and ${\mathfrak f}\in \cush$ be an eigenform of all Hecke operators.
Then for any $\varepsilon>0$, we have
\begin{equation}\label{VoronoiSf}
\begin{aligned}
\mathcal{S}_{\f}(x,a/d) 
& =  \frac{x^{1/4}}{\pi \sqrt{2}}  
\sum_{n\le M} \frac{\lambda(n; d)\phi_a(n,d)}{n^{3/4}} 
\cos\bigg(4\pi\frac{\sqrt{nx}}{q_d}-\frac{\ell+1}{2}\pi\bigg) 
\\
& \quad 
+ O_{{\mathfrak f}, \varepsilon}\big(x^{\varepsilon} d^2(x^{1/2+\varrho} M^{-1/2} + M^{\varrho})\big)
\end{aligned}
\end{equation}
uniformly for $2\le M\le x$ and $1\le d\le x^{1/2}$,
where $\varrho$ is defined as in \eqref{varrho}. 

Moreover for $1\le Q\le x^{1/2}$ and any integer $a$, 
\begin{align*}
\mathcal{S}_{\f}^\mathcal{A}(x) 
& = \frac{x^{1/4}}{\sqrt{2} \pi Q}   
\sum_{d\mid Q}  \sum_{n\le M} \frac{\lambda(n; d)  \phi_a(n,d)}{n^{3/4}}
\cos\bigg(4\pi\frac{\sqrt{nx}}{q_d}-\frac{\ell+1}{2}\pi\bigg) 
\\
& \quad 
+ O\big(x^{\varepsilon} Q (x^{1/2+\varrho} M^{-1/2} + M^{\varrho})\big).
\end{align*}
In particular, for $Q\le x^{\frac12-\varrho}$ and any $a$,  
\begin{equation}\label{MeanValue}
\mathcal{S}_{\f}^\mathcal{A}(x)\ll_{{\mathfrak f}, \varepsilon} Q^{1/3} x^{(1+\varrho)/3+\varepsilon}.
\end{equation}
\end{theorem}
\begin{remark}
It is shown in \cite[Proposition 3.2]{MeherMurty2014} that $\mathcal{S}_{\f}^\mathbb{N}(x)\ll x^{2/5+\varepsilon}$, which is superseded by the particular case $\mathcal{A}=\mathbb{N}$ (and $Q=1$) of \eqref{MeanValue} for $\varrho=1/6+\varepsilon$ is admissible. 
\end{remark}

\begin{proof}
Let $d\le x^{1/2}$, $1\le M\le x$ and $T>1$ be  chosen as 
\begin{eqnarray}\label{T}
T^2= q_d^{-2}4\pi^2(M+1/2)x   \gg 1.
\end{eqnarray}
We apply the Perron formula (cf. \cite[Corollary II.2.2.1]{Tenenbaum1995})
to \eqref{calL} with $\kappa:=1+\varepsilon$, $\sigma_a=\alpha=1$ and $B(n)=C_{\varepsilon}n^{\varrho}$ to write
\begin{equation}\label{PerronFormula}
\mathcal{S}_{\f}(x,a/d) 
= \frac1{2\pi {\rm i}} 
\int_{\kappa-{\rm i}T}^{\kappa+ {\rm i}T} \mathcal{L}_\f(s,a/d) \frac{x^s}{s} \d s 
+ O_{\mathfrak{f}, \varepsilon}\bigg(\frac{dx^{1+\varrho}}{T}\bigg).
\end{equation}

We deform the line of integration to the contour ${\mathscr L}$ joining the points
$\kappa - {\rm i}T$,
$-\varepsilon - {\rm i}T$,
$-\varepsilon + {\rm i}T$,
$\kappa + {\rm i}T$. 
Let ${\mathscr L}_{\rm v}:=[-\varepsilon-{\rm i}T, -\varepsilon+{\rm i}T]$.
By Lemma~\ref{convex}, the integrals over the horizontal segments of ${\mathscr L}$ are $\ll x^\vep (xT^{-1}+d^{3/2})$, and the pole of the integrand at $s=0$ gives $\mathcal{L}_\f(0,a/d)\ll d^{3/2+\vep}$. By  the functional equation \eqref{auxLfe}, the integral over ${\mathscr L}_{\rm v}$ equals
\begin{align*}
\frac1{2\pi {\rm i}} \int_{{\mathscr L}_{\rm v}} \mathcal{L}_\f(s,a/d) \frac{x^s}s \d s
& = q_d \ic^{-(\ell+1/2)} \frac1{2\pi {\rm i}} \int_{{\mathscr L}_{\rm v}} 
\frac{ L_\infty(1-s)}{L_\infty(s)}  \widetilde{\mathcal{L}}_\f(1-s,a/d) 
\bigg(\frac{\sqrt{x}}{q_d}\bigg)^{2s} \frac{\d s}{s}
\end{align*}
By  \eqref{auxLs} and \eqref{eq7.3}, we express \eqref{PerronFormula} into
\begin{equation}\label{3.4}
\mathcal{S}_\f(x,a/d)
=  \frac{\sqrt{q_d}}{2\pi}  \sum_{n\ge 1}\frac{\lambda(n; d)\phi_a(n,d)}{n} I_{{\mathscr L}_{\rm v}}\left(\frac{2\pi\sqrt{nx}}{q_d}\right)
+O\bigg(\frac{dx^{1+\varrho}}{T}+d^{3/2} x^\vep\bigg)
\end{equation}
where
$$
I_{{\mathscr L}_{\rm v}}(y)
:= \frac{1}{2\pi {\rm i}} \int_{{\mathscr L}_{\rm v}} 
\frac{\Gamma(1-s+\ell/2-1/4)}{\Gamma(s+\ell/2-1/4)} \cdot \frac{y^{2s}}s \d s.
$$
Next we apply the stationary phase method to bound $I_{{\mathscr L}_{\rm v}}(y)$ for large $y$ 
and give an asymptotic expansion in terms of trigonometric functions for small $y$.

With Stirling's formula, for $\tau>0$, the integrand equals
$$
{\rm e}^{\text{i}\pi (\ell-1)/2} y^{2\sigma} \tau^{-2\sigma} {\rm e}^{2\text{i}\tau\log(\ee y/\tau)}
\big\{1+c_1\tau^{-1} + O\big(\tau^{-2}\big)\big\}
$$ 
for any $|\tau|\ge 1$ and $|\sigma|\le A$, where $c_1$ and $A>0$ denote some suitable constants 
and the  implied $O$-constant is independent of $\tau$ and $y$. 
Set $g(\tau) := 2\tau\log({\ee} y/\tau)$, then $g'(\tau)= 2\log( y/\tau)$.   
With the second mean value theorem for integrals (cf. \cite[Theorem I.0.3]{Tenenbaum1995}), 
we obtain for $ y>T$ and $\sigma =-\vep$, 
\begin{equation}\label{est1}
\int_1^T y^{2\sigma}\tau^{-2\sigma} \text{e}^{\text{i}g(\tau)} \big\{1+c_1\tau^{-1} + O\big(\tau^{-2}\big)\big\} \d \tau
\ll T^{2\vep} y^{2\sigma}\left|\log \frac{ y}T\right|^{-1}+T^{2\vep-1}y^{2\sigma},
\end{equation}
and for $ y<T$ and $\sigma =\frac12+\vep$,
\begin{equation}\label{est2}
\int_T^\infty y^{2\sigma} \tau^{-2\sigma} \text{e}^{\text{i}g(\tau)} \big\{1+c_1\tau^{-1} + O\big(\tau^{-2}\big)\big\} \d \tau
\ll 
 T^{-1-2\vep}y^{2\sigma}\left|\log \frac{ y}T\right|^{-1}+T^{-1-2\vep}y^{2\sigma}.
\end{equation}
For $n>M$, we infer by \eqref{est1} that
\begin{equation*}
\begin{aligned}
I_{{\mathscr L}_{\rm v}}\bigg(\frac{2\pi\sqrt{nx}}{q_d}\bigg)
& \ll_k \bigg(\frac{x}{\sqrt{n}}\bigg)^{2\varepsilon} 
\bigg(\bigg|\log \frac{n}{M+1/2} \bigg|^{-1}+d(Mx)^{-1/2} \bigg).
\end{aligned}
\end{equation*}
By  $\lambda(n; d)\ll n^{\varrho+\vep}$ from Lemma~\ref{lem3.1} and  $|\phi_a(n,d)|\le \sqrt{2}d^{3/2}$, it follows that
\begin{align*}
\sqrt{q_d} \sum_{n>M} \frac{|\lambda(n; d)\phi_a(n,d)|}{n^{1+\varepsilon}} \left|\log \frac{n}{M+1/2}\right|^{-1}
& \ll d^{2} M^\varrho \sum_{M<n<2M} |n-(M+1/2)|^{-1}
\\
& \ll d^{2} M^{\varrho+\vep}.
\end{align*}
Consequently we deduce that
\begin{equation}\label{tail}
\frac{\sqrt{q_d}}{2\pi} \sum_{n>M}\frac{\lambda_{\fh}(n)\phi_a(n,d)}{n} 
I_{{\mathscr L}_{\rm v}}\bigg(\frac{2\pi\sqrt{nx}}{q_d}\bigg) 
\ll x^{\varepsilon} d^2 M^{\varrho} + x^\vep d^2 (Mx)^{-1/2} .
\end{equation}

For $n\le M$, we complete the path ${\mathscr L}_{\rm v}$ to the contour ${\mathscr L}_{\rm v}^*$ 
so as to apply \cite[Lemma~1]{CN1963}, where ${\mathscr L}_{\rm v}^*$ is the positively oriented contour consisting of 
${\mathscr L}_{\rm v}$, ${\mathscr L}_{\rm v}^\pm$ and ${\mathscr L}_{\rm h}^\pm$ with
$$
{\mathscr L}_{\rm v}^\pm 
:= [\tfrac{1}{2}+\varepsilon \pm {\rm i}T, \, \tfrac{1}{2}+\varepsilon \pm {\rm i}\infty),
\qquad
{\mathscr L}_{\rm h}^\pm
:= [-\varepsilon \pm {\rm i}T, \, \tfrac{1}{2} + \varepsilon \pm {\rm i}T].
$$ 
Correspondingly we denote by  $I_{{\mathscr L}_{\rm v}^\pm}$ and  $I_{{\mathscr L}_{\rm h}^\pm}$ the integrals over these segments. 
By \eqref{est2}, the integral over the vertical line segments ${\mathscr L}_{\rm v}^\pm$ is  
\begin{align*}
I_{{\mathscr L}_{\rm v}^\pm}
& \ll x^\varepsilon \bigg(\frac{n}{M}\bigg)^{1/2} \bigg|\log \frac{n}{M+1/2}\bigg|^{-1} ,
\end{align*}
while for the horizontal segments, $I_{{\mathscr L}_{\rm h}^\pm}$ contributes at most $O((n/M)^\vep)$. Thus 
\begin{equation}\label{sumILhpm}
\begin{aligned}
&\frac{\sqrt{q_d}}{2\pi} \sum_{n\le M} \frac{\lambda(n; d)\phi_a(n,d)}{n} 
\left(I_{{\mathscr L}_{\rm v}^\pm}+I_{{\mathscr L}_{\rm h}^\pm}\right)
\\
& \ll  x^\vep d^2M^{\rho -1/2} \sum_{M/2\le n\le M} n^{-1/2}  \bigg|\log \frac{M+1/2}{M+1/2-n}\bigg|^{-1} 
\\
& \ll x^\vep d^2M^{\varrho}. 
\end{aligned}
\end{equation}
Inserting \eqref{sumILhpm} and \eqref{tail} into \eqref{3.4},  we get from our choice of $T$,
\begin{equation}\label{610}
\begin{aligned}
\mathcal{S}_\f(x,a/d) 
& = \frac{\sqrt{q_d}}{2\pi}  \sum_{1\le n\le M}\frac{\lambda(n; d)\phi_a(n,d)}{n} I_{{\mathscr L}_{\rm v}^*}\left(\frac{2\pi\sqrt{nx}}{q_d}\right) 
\\
& \quad 
+ O\big(x^\vep d^{2}(x^{1/2+\varrho}M^{-1/2}+M^\rho)\big).
\end{aligned}
\end{equation}

Now all the poles of the integrand in 
$$
I_{{\mathscr L}_{\rm v}^*}(y)
:= \frac{1}{2\pi {\rm i}} \int_{{\mathscr L}_{\rm v}^*}
\frac{\Gamma(1-s+\ell/2-1/4) \Gamma(s)}{\Gamma(s+\ell/2-1/4) \Gamma(s+1)} 
y^{2s}  \d s.
$$
lie on the right of the contour ${\mathscr L}_{\rm v}^*$. 
After a change of variable $s$ into $1-s$, we have 
\begin{align*} 
I_{{\mathscr L}_{\rm v}^*}(y) 
& = \frac{1}{\pi} I_0\big(y^2\big),
\end{align*}
with 
$$
I_0(y)
:= \frac1{2\pi {\rm i}}\int_{{\mathscr L}_\varepsilon}  
\frac{\Gamma(s+(2\ell-1)/4)\Gamma(1-s)}{\Gamma(1-s+(2\ell-1)/4)\Gamma(2-s)} 
y^{1-s} \d s. 
$$ 
Here ${\mathscr L}_\varepsilon$ consists of the line $s=\dm-\varepsilon+{\rm i}\tau$ with $|\tau|\ge T$,
together with three sides of the rectangle whose vertices are 
$\dm-\varepsilon-{\rm i}T$,
$1+\varepsilon-{\rm i}T$,
$1+\varepsilon-{\rm i}T$
and
$\dm-\varepsilon+{\rm i}T$.
Clearly our $I_0$ is a particular case of $I_\rho$ defined in \cite[Lemma 1]{CN1963},
corresponding to the choice of parameters 
$A=\delta=N=\omega=\alpha_1=1$, 
$\beta_1=\mu=(\ell-2)/4$,
$\rho=m=0$,
$a=-\frac{3}{4}$,
$c_0=\frac{1}{2}$,
$h=2$,
$k_0=-(\ell+1)/2$.
It hence follows that
\begin{equation}\label{3.9}
I_{{\mathscr L}_{\rm v}^*}\bigg(\frac{2\pi\sqrt{nx}}{q_d}\bigg)
= {e_0'} \sqrt{\frac{2\pi}{q_d}}  (nx)^{1/4}
\cos\bigg(4\pi\frac{\sqrt{nx}}{q_d}-\frac{\ell+1}{2}\pi\bigg) 
+ O\big(d^{1/2} (nx)^{-1/4}\big).
\end{equation}
The value of $e_0'$ \cite[Lemma 1]{CN1963} is $1/\sqrt{\pi}$, and the main term in \eqref{VoronoiSf} follows from \eqref{3.9} and \eqref{610}. With a simple checking, the  $O$-term in \eqref{3.9} gives a term that will be absorbed in \eqref{610}.

Finally we set  $M=Q^{4/3}x^{(1+4\rho)/3}$ and note from \eqref{eq7.3} that
\begin{eqnarray*}
\sum_{n\le M} \frac{|\lambda(n; d)\phi_a(n,d)|}{n^{3/4}}
\ll d^{1+\vep} \sum_{n\le M} |\lambda(n; d)|^2n^{-3/4} + d^{1+\vep} \sum_{n\le M} (n,d) n^{-3/4},
\end{eqnarray*}
which is  $\ll x^\vep d  M^{1/4}$ with \eqref{msq}.
\end{proof}

\vskip 8mm

\section{Preparation for the proof of Theorem \ref{thm2}}\label{pthm2}

We consider odd $Q$ only, then $q_d=2d$ and $\lambda(n; d)=\lambda_\fh(n)$ for all $d\mid Q$. 
The idea of proof is the same as in Heath-Brown \& Tsang \cite{HBT94}, 
however,  some new technicality arises because of the new frequencies ($\sqrt{n}/q_d$ rather than $\sqrt{n}$). Consequently, instead of  $\sqrt{1}$, we shall apply their argument to the frequency $\sqrt{n_0}/Q$ 
where $n_0=2^jf_0$ with $j\ge 0$ and $f_0$ squarefree, 
and simultaneously, require the coefficient $\lambda_\fh(n_0) \phi_a(n_0,Q)$ to be non-vanishing.  
We can guarantee the existence of $n_0$ under certain circumstances. 

For convenience, let us recall our notation (specialized to this case $2\nmid d$): 
\begin{equation*}
\mathcal{S}_\f^\mathcal{A}(x) 
=  \sum_{\substack{n\le x\\ n\equiv a ({\rm mod}\,Q)}} \lambda_{\f}(n) 
\quad \mbox{ and }\quad
\mathcal{S}_{\f}(x,a/d) := \sum_{n\le x} \lambda_{\f}(n) R_d(n-a).
\end{equation*} 
where $R_d(m)= \sumstar_{\hskip -1,5mm u ({\rm mod}\, d)} \ee\left(mu/d\right)$ is the Ramanujan sum.
Their associated Dirichlet series are
\begin{equation*}
L_\f(s,a,Q) := \sum_{\substack{n\ge 1\\ n\equiv a ({\rm mod} \, Q)}} \lambda_{\f}(n) n^{-s}
\quad \mbox{ and }\quad
\mathcal{L}_\f(s,a/d):= \sum_{n\ge 1} \lambda_\f(n)R_d(n-a)n^{-s}.
\end{equation*}
Moreover, $L_\f(s,a,Q) = Q^{-1} \sum_{d\mid Q} \mathcal{L}_\f(s,a/d)$ and 
\begin{eqnarray*}
(2d)^{s} L_\infty(s) \mathcal{L}_\f(s,a/d) = \ic^{-(\ell+1/2)} (2d)^{1-s} L_\infty(1-s)  \widetilde{\mathcal{L}}_\f(1-s,a/d) 
\end{eqnarray*}
where 
\begin{eqnarray*}
\widetilde{\mathcal{L}}_\f(s,a/d) 
:=  \sum_{n\ge 1} \lambda_{\fh}(n) {\rm K}(a, n; d) n^{-s} .
\end{eqnarray*}

\begin{lemma}\label{lem7.1}  Under the assumption that $\{\lambda_\f (n)\}_{n\in \N}$ is a real  sequence,   for all $a,d$, the sequences $\{\ic^{-(\ell+1/2)}\lambda_\fh(n){\rm K}(a, n; d)\}_{n\in\N}$ are real.
\end{lemma} 
\begin{proof} Since the Ramanujan sum $R_d(m)$ is real-valued, $\mathcal{L}_\f(s,a/d)$  is real-valued for $s\in (1,\infty)$ under the given assumption. The holomorphicity of $\mathcal{L}_\f(s,a/d)$ implies that $\overline{\mathcal{L}_\f(\overline{s},a/d)}$ is holomorphic. Thus $\overline{\mathcal{L}_\f(\overline{s},a/d)} =\mathcal{L}_\f(s,a/d)$ on $\C$ (as they are equal on $(1,\infty)$). The lemma follows.
\end{proof}

\begin{lemma}\label{lem7.2} 
When the sequence $\{\lambda_\f(n)\}_{n\in \mathcal{A}}$ contains nonzero terms, 
the function  $\mathcal{L}_\f (s,a/d)$  is non-identically zero for all $d\mid Q$. 
\end{lemma}

\begin{proof} 
Suppose not, say, $\mathcal{L}_\f (s,a/d_0)\equiv 0$. Then 
\begin{equation*}
\sum_{\substack{n\ge 1\\ n\equiv a ({\rm mod} \, Q)}} \lambda_{\f}(n) n^{-s} 
= Q^{-1} \sum_{\substack{d\mid Q\\ d\neq d_0}} \mathcal{L}_\f(s,a/d)
= \sum_{n\ge 1} n^{-s} \lambda_\f(n) Q^{-1} \sum_{\substack{d\mid Q\\ d\neq d_0}} R_d(n-a).
\end{equation*}
With  the standard formula for the Ramanujan sum, we infer that 
\begin{eqnarray*}
\delta_{n\equiv a ({\rm mod}\,Q)} \lambda_\f(n) 
= \lambda_\f(n) Q^{-1} \sum_{\substack{d\mid Q\\ d\neq d_0}} 
\sum_{\substack{\delta\mid d\\ (d/\delta)\mid (n-a)}} \mu(\delta) (d/\delta) 
\qquad\mbox{$\forall$ $n\ge 1$}.
\end{eqnarray*}
Take $n\equiv a \,(\bmod\,{Q})$ such that $\lambda_\f(n)\neq 0$. We obtain that
$$
Q-\phi(d_0)
= \sum_{\substack{d\mid Q\\ d\neq d_0}} \phi(d)
= \sum_{\substack{d\mid Q\\ d\neq d_0}} \sum_{\delta\mid d} \mu(\delta) (d/\delta) = Q.
$$
Contradiction arises.
\end{proof}

\begin{proposition}\label{propsign} 
Let $Q\ge 1$ be odd and $0\le a< d$. Suppose $n_0=2^jf_0$ with $f_0$ squarefree  and $j\ge 0$ is an integer such that
\begin{equation}\label{propcond}
\lambda_\fh(n_0) \phi_a(n_0,Q)\neq 0.
\end{equation}
Then there are constants $c_0=c_0(\f,Q,n_0)$ and $x_0=x_0(\f, Q,n_0)$ such that $\mathcal{S}_\f^\mathcal{A}(x)$ attains at least one sign change in the interval $[x,x+c_0\sqrt{x}]$ for all $x\ge x_0$. 
\end{proposition}
\begin{proof}
Let $\alpha$ a parameter determined later and $T$ be any sufficiently large number. Set 
$$
F_{\mathfrak{f}}(t+\alpha u)
:= \pi \sqrt{Q} \frac{S_{\f}^\mathcal{A}((Q(t+\alpha u))^2)}{\sqrt{t+\alpha u}}\qquad \mbox{ ($t\in [T,2T]$, $u\in [-1,1]$)}.  
$$
By Theorem~\ref{Voronoi} with $M=(QT)^2$, we deduce that
\begin{equation}\label{F}
\begin{aligned}
F_{\mathfrak{f}}(t+\alpha u) 
& = \sum_{d\mid Q} \sum_{n\le (QT)^2}  \frac{\lambda_{\fh}(n)\phi_a(n,d) }{n^{3/4}} 
\cos\bigg(\pi (t+\alpha u)\frac{Q\sqrt{n}}d -\frac{\ell+1}{2}\pi\bigg) \nonumber\\
&\quad + O\big( Q (QT)^{2\varrho-1/2+\vep }\big).
\end{aligned}
\end{equation}

Let  $\tau =1$ or $-1$, and define 
$$
k_\tau(u)
: = (1-|u|)(1+\tau \cos(2\pi\alpha\sqrt{n_0} u)).
$$ 
Then as in the proof of \cite[Lemma 3.2]{LW2009},  for any $n\in \N$ and $t\in \R$, the integral
$$
r_n
=r_n(\alpha, \tau, t)
:=\int_{-1}^1 k_\tau(u) \cos\bigg(2\pi(t+\alpha u)\frac{Q\sqrt{n}}d -\frac{\ell+1}{2}\pi\bigg)\d u
$$
satisfies
\begin{equation}\label{3.11}
\begin{aligned}
r_n
& = \delta_{Q\sqrt{n}= d\sqrt{n_0}}\cdot \frac{\tau}{2}\cos\bigg(2\pi t\sqrt{n_0}-\frac{\ell+1}{2}\pi\bigg) \\
& \quad + O\bigg(\min\bigg(1,\frac1{\alpha^2n}\bigg)
+\delta_{Q\sqrt{n}\neq d\sqrt{n_0}}\min\bigg(1,\frac1{(\alpha_{n,d}^{-})^2}\bigg)\bigg),
\end{aligned}
\end{equation}
where $\alpha_{n,d}^-=\alpha |Q\sqrt{n}-d\sqrt{n_0}|/d$, $\delta_*=1$ if $*$ holds, or $0$ otherwise. The $O$-constant is absolute.

Observe that $Q\sqrt{n}=d\sqrt{n_0}$ if and only if $2^jf_0=(Q/d)^2n$ 
which is equivalent to $n=2^jf_0=n_0$ and $d=Q$ since $f_0$ is squarefree and $Q/d$ is odd. 
Following from \eqref{F} and \eqref{3.11},  the integral 
$$
J_\tau (t)
= \int_{-1}^1 F_{\mathfrak{f}}(t+\alpha u) k_\tau(u)\d u
$$
can be written as
\begin{equation}\label{J}
J_\tau (t) 
= \frac{\tau}2  \frac{\lambda_{\fh}(n_0)\phi_a(n_0,Q) }{n_0^{3/4}}\cos\bigg(2\pi t\sqrt{n_0}-\frac{\ell+1}{2}\pi\bigg)
+ {\rm E} + O\big(   Q (QT)^{2\varrho-1/2+\vep }\big)
\end{equation}
where 
\begin{eqnarray*}
{\rm E} \ll \frac1{\alpha^2}
\sum_{d\mid Q} \sum_{n\le (QT)^2}  \frac{|\lambda_{\fh}(n)\phi_a(n,d)| }{n^{7/4}} 
+ \sum_{d\mid Q} \frac{d^2}{\alpha^2} 
\sum_{\substack{n\le (QT)^2\\ Q\sqrt{n}\neq d\sqrt{n_0}}}  
\frac{|\lambda_{\fh}(n)\phi_a(n,d)| }{n^{3/4}|Q\sqrt{n}-d\sqrt{n_0}|^2}\cdot
\end{eqnarray*}
Using the bounds $\phi_a(n,d)\ll d^{3/2}$ and $\lambda_\fh(n)\ll n^{\varrho}$, a little calculation gives
$$
{\rm E}\ll Q^{3}n_0^{\varrho+1/4}\alpha^{-2}.
$$

Let $A_0:= |\lambda_{\fh}(n_0)\phi_a(n_0,Q) |n_0^{-3/4}$, which is $>0$. Fix a sufficiently large $\alpha=\alpha(\f,n_0,Q)$, so that $E$ is $<\frac1{8} A_0$, and then a sufficiently large $T_0=T_0(\f, n_0,Q, \alpha)$ such that the $O$-term $O\big(   Q (QT)^{2\varrho-1/2+\vep }\big)$ is $\le \frac18 A_0$ for all $T\ge T_0$. Now observe that for any $m\in\N$, the absolute value of the cosine  factor is $1/\sqrt{2}$ if $t=t_m$ where 
$$
t_m:= (m+\tfrac18)n_0^{-1/2}.
$$
This implies $|J_\tau(t_m)|> \tfrac{1}{4}(\sqrt{2}-1)A_0>0$ whenever $t_m>T_0+\alpha$. 
Since $J_\pm(t_m)$ are of opposite signs and the kernel function $k_\tau$ is nonnegative,  
there is  a pair of $t_m^\pm\in [t_m-\alpha, t_m+\alpha]$ for which $\pm F_\f (t_m^\pm)>0$. 
Equivalently, $\mathcal{S}_\f^\mathcal{A}(y)$ attains a sign change in every interval of the form 
$[(Q(t_m-\alpha))^2, (Q(t_m+\alpha))^2]$ whose length is $\ll \alpha (Q^2 t_m)\ll_{\f,Q,n_0} \sqrt{x}$ 
when $x=(Qt_m)^2$. Our result follows readily. 
\end{proof}

\vskip 8mm

\section{Proof of Theorem \ref{thm2}}\label{Sthm2}

In view of Proposition~\ref{propsign}, the main task is to study the condition $\lambda_\fh(n_0) \phi_a(n_0,Q)$. Recall $\phi_a(n,Q)= \sqrt{2Q} \ic^{-(\ell+1/2)}{\rm K}(a, n; Q)$ by \eqref{eq7.3}. Clearly, $\phi_a(n,1)= \sqrt{2}$. In general, we have by Lemma~\ref{lem9.1} (2),
\begin{equation}\label{eqphi}
\phi_a(n,Q)= \sqrt{2Q} \, \vep_Q^{-(2\ell+1)} \prod_{p^\alpha\|Q} S(n\overline{4 Q_p}, a\overline{Q_p}; p^\alpha)
\end{equation}
where $S(m, n; c)$ is defined as in \eqref{def:Smnc}, 
$Q_p= Q/p^\alpha$ and $\overline{x}x \equiv 1\,(\bmod\,{p^\alpha})$  
for each term inside the product, $\forall$ $p^\alpha\,\|\,Q$.  
\begin{itemize}[leftmargin=5mm]
\renewcommand\labelitemi{$\spadesuit$}
\item 
Case 1. $Q=1$.  It suffices to find a squarefree $t$ and a  $j\ge 0$ such that $\lambda_\fh(2^jt)\neq 0$. By Lemma~\ref{lem7.2}, $\mathcal{L}_\f (s,1)$ and thus $\widetilde{\mathcal{L}}_\f (s,1)= \sum_{n\ge 1} \lambda_\fh(n)n^{-s}$ are not identical to the zero function. Thus $\lambda_\fh(n)\neq 0$ for some $n\in\N$. Write $n=2^jtm^2$ where $t$ is squarefree and $m$ is odd, $\lambda_\fh(2^jt)\neq 0$ from \eqref{mul}. 

\item 
Case 2.  
$a=0$ and $p^\alpha\,\|\,Q$ implies $\alpha$ being odd.  By Lemma~\ref{lem9.1} (2)-(3) and \eqref{eqphi}, $\phi_0(n,Q)=0$ if $(n,Q)>1$. Repeating the argument in Case 1, we get $\lambda_\fh(n) \phi_0(n,Q)\neq 0$ for some $n\in\N$. This $n$ has to be coprime with $Q$. Write $n=2^jtm^2$ with squarefree $t$ and odd $m$, then $\lambda_\fh(2^jt)\neq 0$ (from $\lambda_\fh(2^jtm^2)\neq 0$) and $\phi_0(2^jt,Q)\neq 0$ because
$$
S(hk, 0; Q)= \left(\frac{h}Q\right) S(k, 0; Q)
$$
 if $(h,Q)=1$, from the definition of the Sali\'e sum.

\item
Case 3. $(a, Q)=1$ and $p^2\mid Q$, $\forall$ $p|Q$. The argument is similar to the previous cases -- firstly finding $n=2^jtm^2$, with squarefree $t$ and odd $m$, for which $\lambda_\fh(n) \phi_0(n,Q)\neq 0$. But now we need  \eqref{Kform} to analyze the Sali\'e sum, which gives
$$
\phi_a(2^jtm^2,Q)= \sqrt{2}Q \vep_Q^{-2\ell}  \bigg(\frac{a}Q\bigg) c_{a2^jt} (m, Q)
$$
where
\begin{equation}\label{cbm}
c_b(m, d) 
= \sum_{\substack{y \,({\rm mod}\, d)\\ y^2\equiv bm^2 ({\rm mod}\, d)}} \ee\bigg(\frac{y}{d}\bigg).
\end{equation}
 
As in \eqref{eqphi}, we have the factorization
$$
c_{a2^jt} (m,Q) = \prod_{p^\alpha\|Q} c_{\overline{Q_p}a2^jt}(m, p^\alpha)
$$
and the lemma below assures $(m,Q)=1$ and $\phi_a(2^jt,Q)\neq 0$ when $\phi_a(2^jtm^2,Q)\neq 0$. Hence this case is also complete. 
\end{itemize}

\begin{lemma}\label{lem8.1} 
Let $b\in\Z$, $p$ an odd prime and $\alpha\ge 2$. Define $c_b(m,p^\alpha)$ as in \eqref{cbm}. 
Then 
\begin{itemize}
\item[{\rm (i)}]
$c_b(m,p^\alpha)=0$ if $p\mid m$, 
and 
\item[{\rm (ii)}] 
$c_b(1,p^\alpha)\neq 0$ if $c_b(m,p^\alpha)\neq 0$ with $p\nmid m$.
\end{itemize}
\end{lemma}

\begin{proof} 
(i) 
Write $m=p^\beta m'$ where $p\nmid m'$. 
\begin{itemize}[leftmargin=10mm]
\item 
$\alpha=2\gamma \le 2\beta$. Then 
\begin{eqnarray*}
c_b(m,p^\alpha)
=  \sum_{y^2\equiv 0 ({\rm mod}\, p^\alpha)} \ee\bigg(\frac{y}{p^\alpha}\bigg)
=  \sum_{l ({\rm mod}\,p^\gamma)} \ee\bigg(\frac{l}{p^\gamma}\bigg) =0.
\end{eqnarray*}
\item 
$\alpha=2\gamma+1 \le 2\beta$. Then $y$ is of the form $y=lp^{\gamma+1}$, and as $\gamma\ge 1$,
\begin{eqnarray*}
c_b(m,p^\alpha)
=  \sum_{y^2\equiv 0 ({\rm mod}\, p^\alpha)} \ee\bigg(\frac{y}{p^\alpha}\bigg)
=  \sum_{l ({\rm mod}\,p^\gamma)} \ee\bigg(\frac{l}{p^\gamma}\bigg) =0.
\end{eqnarray*}
\item 
$\alpha>2\beta\ge 2$. Then $y=lp^\beta$ and thus
\begin{align*}
c_b(m,p^\alpha)
& =   
\sum_{l^2\equiv bm'^2 ({\rm mod}\, p^{\alpha-2\beta})} 
\sum_{y \equiv p^\beta l ({\rm mod}\, p^\alpha)} \ee\bigg(\frac{y}{p^\alpha}\bigg)
\\
& =   
\sum_{l^2\equiv bm'^2 ({\rm mod}\, p^{\alpha-2\beta})} 
\sum_{t ({\rm mod}\, p^\beta)} \ee\bigg(\frac{l +tp^{\alpha-2\beta}}{p^{\alpha-\beta}}\bigg)
\\
& = \sum_{l^2\equiv bm'^2 ({\rm mod}\, p^{\alpha-2\beta})} \ee\bigg(\frac{l}{p^{\alpha-\beta}}\bigg)
\sum_{t ({\rm mod}\, p^\beta)} \ee\bigg(\frac{t}{p^{\beta}}\bigg)
\\\noalign{\vskip 1mm}
& = 0.
\end{align*}
\end{itemize}

(ii) 
Suppose $c_b(m,p^\alpha)\neq 0$ where $(m,p)=1$. 
We may assume $p^2\nmid b$, for otherwise, $c_b(m,p^\alpha)= c_{b/p^2}(mp, p^\alpha)=0$ by (i). 
Also $p\,\|\,b$ cannot happen because,  
when $\alpha\ge 2$, $p^2\mid b$ if $p\mid b$ and $y^2\equiv b m^2 (\bmod\,{p^\alpha})$ has solutions. 
Thus $p\nmid b$. 

Now $c_b(m,p^\alpha)\neq 0$ implies the congruence  $y^2\equiv b m^2\,(\bmod\,{p^\alpha})$ is soluble, 
and with $(m,p)=1$, $y^2\equiv b\,(\bmod\,{p^\alpha})$ has two solutions, say, $\pm y_0$ and $p\nmid y_0$.  
We see that 
$$
\sum_{y^2\equiv b ({\rm mod}\, p^\alpha)} \ee\bigg(\frac{y}{p^\alpha}\bigg) 
= 2\cos\bigg(2\pi \frac{y_0}{p^\alpha}\bigg)
\neq 0
$$
because otherwise, $y_0/p^\alpha = (2r+1)/4$ for some $r\in \Z$ or equivalently, $4y_0 = (2r+1)p^\alpha$ which contradicts to $p\nmid y_0$. 
\end{proof}

\vskip 8mm

\section{Appendix}

Let us denote, as in \cite[Section 3]{Iwaniec1987},  the Kloosterman-Sali\'e sum by
\begin{eqnarray*}
K_{2\ell+1}(m,n; c) 
:= \sum_{d\,({\rm mod}\,c)} \vep_d^{-(2\ell+1)} \bigg(\frac{c}d\bigg) \ee\bigg(\frac{m d+n\overline{d}}c\bigg)
\end{eqnarray*}
and 
\begin{equation}\label{def:Smnc}
S(m, n; c) := \sum_{x\, ({\rm mod}\, c)} \bigg(\frac{x}c\bigg) \ee\bigg(\frac{m x+n\overline{x}}c\bigg),
\end{equation}
where $c\in \N$ and $m,n\in\Z$. Then we have the following estimate,
\begin{align}\label{eq:9.1}
|K_{2\ell+1}(n, m; d)| 
\qquad 
\mbox{and} 
\qquad 
|S(m, n; d)| \le d^{1/2}\tau(d)(d, n, m)^{1/2}
\end{align} 
where $\tau(n)$ is the divisor function. This follows from the well-known Weil's bound for Kloosterman sums and the following lemma.

\begin{lemma}\label{lem9.1}
We have the following results:
\begin{enumerate}
\item[$\mathrm{(a)}$] 
Let $c=qr$ with $r\equiv 0\,(\bmod\,{4})$ and $(q,r)=1$. Then 
$$
K_{2\ell+1}(m,n; c)= K_{2\ell+2-q}(m\overline{q}, n\overline{q}; r) S(m\overline{r}, n\overline{r}; q)
$$
where $q\overline{q}\equiv 1\,(\bmod\,{r})$ and $r\overline{r}\equiv 1\,(\bmod\,{q})$.
\item[$\mathrm{(b)}$] 
Let $q$ be odd, $q=uv$ with $(u,v)=1$. Then 
$$
S(m, n; q) = S(m\overline{u}, n\overline{u}; v) S(m\overline{v}, n\overline{v}; u)
$$
where $u\overline{u}\equiv 1\,(\bmod\,{v})$ and $v\overline{v}\equiv 1\,(\bmod\,{u})$.
\item[$\mathrm{(c)}$]  
For an odd prime $p$ and odd $\alpha$, if $p\mid m$, then $S(m, 0; p^\alpha)=0$.
\item[$\mathrm{(d)}$]  
If $(c, 2)=1$, then $|S(m, n; c)|\le (m, n, c)^{1/2}c^{1/2} \tau(c)$. 
\item[$\mathrm{(e)}$] 
Let $4|r|2^\infty$. Then $|K_{2\ell+1} (m, n; r)|\le (m, n, r)^{1/2}r^{1/2} \tau(r)$.
\end{enumerate}
\end{lemma}

\begin{proof} 
(a) 
See \cite[p. 390, Lemma 2]{Iwaniec1987}.

(b)
See \cite[p. 390, Lemma 3]{Iwaniec1987}.  

(c)
By definition, for odd $\alpha$, we have
$$
S(m, 0; p^\alpha) 
= \sum_{x \,({\rm mod}\, p^\alpha)} \left(\frac{x}p\right)\ee\bigg(\frac{mx}{p^\alpha}\bigg).
$$
When $\alpha=1$, $S(m, 0; p^\alpha) =\sum_{x \,({\rm mod}\, p^\alpha)} \big(\frac{x}p\big)=0$ as $p\mid m$. 
Suppose $\alpha\ge 3$. Putting $x=lp +v$, we get
\begin{eqnarray*}
\sum_{l\,({\rm mod}\,p^{\alpha-1})} \ee\bigg(\frac{ml}{p^{\alpha-1}}\bigg) 
\sum_{v\,({\rm mod}\,p)} \bigg(\frac{v}p\bigg)\ee\bigg(\frac{mv}{p}\bigg)= 0.
\end{eqnarray*}

(d)
Iwaniec \cite[Section 4.6]{Iwaniec1997} handled the case $(c,2n)=1$, and thus $(c,2m)=1$ too by symmetry. Together with (b), it suffice to deal with $p\mid (m,n)$ and $c$ is a power of $p$. 

Consider $S:=S(p^a m, p^{a+b} n; p^{a+t})$ where $b\ge 0$, $p\nmid mn$, $a, t\ge 1$ and $a+t$  is odd. (The case that $a+t$ is even is done with the classical Kloosterman sum.) Clearly, 
$$
S =  \sum_{d\,({\rm mod}\, p^{a+t})} 
\bigg(\frac{d}p\bigg) \ee\bigg(\frac{m d+p^b n\overline{d}}{p^t}\bigg) 
= \left(\frac{m}p\right) 
\sum_{d\,({\rm mod}\, p^{a+t})} \bigg(\frac{d}p\bigg) \ee\bigg(\frac{d+p^b mn\overline{d}}{p^t}\bigg).
$$
Mimicking Iwaniec's proof in \cite[p. 67]{Iwaniec1987} (in fact attributed to Sarnak), we consider
$$
F(x) = \sum_{d\,({\rm mod}\, p^{a+t})} \bigg(\frac{d}p\bigg) \ee\bigg(\frac{x^2d+p^b mn\overline{d}}{p^t}\bigg).
$$
and its Fourier transform 
$$
\widehat{F}(y)= \sum_{x\,({\rm mod}\, p^t)} F(x) \ee\!\left(-\frac{xy}{p^t}\right).
$$
As in \cite[p. 67]{Iwaniec1987}, we obtain $\widehat{F}(y)= g(1,p^t) G_t(4mnp^b-y^2)$ where
$$
G_t(4mnp^b-y^2)=\sum_{d\,({\rm mod}\, p^{a+t})} \bigg(\frac{d}p\bigg)^{t+1} \ee\bigg(\frac{d(4mnp^b-y^2)}{p^t}\bigg).
$$

\underline{Case 1: $t$ is odd.} 
Then 
\begin{align*}
G_t(4mnp^b-y^2) 
& = \sumstar_{d\,({\rm mod}\, p^{a+t})}  \ee\bigg(\frac{d(4mnp^b-y^2)}{p^t}\bigg)
\\
& = \sum_{r=0,1} (-1)^r p^a \sum_{d\,({\rm mod}\, p^{t-r})} \ee\bigg(\frac{d(4mnp^b-y^2)}{p^{t-r}}\bigg).
\end{align*}
Since 
$$
\sum_{d\,({\rm mod}\, p^{t-r})} \ee\bigg(\frac{d(4mnp^b-y^2)}{p^{t-r}}\bigg)=p^{t-r} \delta_{y^2\equiv 4mnp^b\, ({\rm mod} \, p^{t-r})},
$$
we conclude
$$
\widehat{F}(y)= g(1,p^t) 
\sum_{r=0,1} (-1)^r p^{a+t-r} \delta_{y^2\equiv 4mnp^b \, ({\rm mod} \, p^{t-r})}
$$
and
\begin{align*}
F(x) 
& = p^{-t} \sum_{y\,({\rm mod}\,p^t)} \widehat{F}(y) \ee\bigg(\frac{xy}{p^t}\bigg)
\\
& = g(1,p^t) \sum_{r=0,1} (-1)^r p^{a-r}
\sum_{\substack{y\,({\rm mod}\,p^t)\\ y^2\equiv 4mnp^b ({\rm mod} \, p^{t-r})}} \ee\bigg(\frac{xy}{p^t}\bigg).
\end{align*}
As $|g(1,p^t)|\le p^{t/2}$ by \cite[(4.43)]{Iwaniec1997}, we see that $|F(1)|\le 2p^{a+t/2}$. 

\vskip2mm

\underline{Case 2: $t$ is even.}
Then 
\begin{align*}
G_t(4mnp^b-y^2) 
& = \sum_{d\,({\rm mod}\, p^{a+t})} \bigg(\frac{d}p\bigg) \ee\bigg(\frac{d(4mnp^b-y^2)}{p^t}\bigg)
\\
& = \sum_{u\,({\rm mod}\, p^{a+t-1})} \ee\bigg(\frac{u(4mnp^b-y^2)}{p^{t-1}}\bigg) 
\sum_{v\,({\rm mod}\,p)} \bigg(\frac{v}p\bigg) \ee\bigg(\frac{v(4mnp^b-y^2)}{p^{t-1}}\bigg).
\end{align*}
The first sum does not vanish only when $y^2\equiv 4mn$ $({\rm mod} \, p^{t-1})$, but in this case, the second sum equals zero. i.e. $G_t(4mnp^b-y^2)= 0$. So $\widehat{F}(y)= g(1,p^t) G_t(4mnp^b-y^2)=0$, implying $F(x)=0$.

(e)
Refer to \cite{DeDeo}, cf. \cite[Section 14]{CZ}.
\end{proof}

{\bf Acknowledgments}.
Lau is supported by GRF 17302514 of  the Research Grants Council of Hong Kong.
L\"u is supported in part by the key project of the National Natural Science Foundation of China (11531008) and IRT1264.
The preliminary form of this paper was finished during the visit of E. Royer and J. Wu at The University of Hong Kong in 2015. 
They would like to thank the department of mathematics for hospitality and excellent working conditions.

\end{document}